\documentclass [11pt]{article}
\usepackage{amssymb,amsmath,comment,amsthm}

\def\N{\mathbb{N}}

\def\F{\mathbb{F}}

\newtheorem{theorem}{Theorem}[section]
\newtheorem{proposition}[theorem]{Proposition}
\newtheorem{corollary}[theorem]{Corollary}
\newtheorem{lemma}[theorem]{Lemma}
\newtheorem{definition}[theorem]{Definition}
\newtheorem{remark}[theorem]{Remark}
\newtheorem{remarks}[theorem]{Remarks}
\newtheorem{conjecture}[theorem]{Conjecture}

\begin{document}
\title{On odd prime divisors of binary perfect polynomials}
\author{Gallardo Luis H. - Rahavandrainy Olivier \\
Universit\'e de Brest, UMR CNRS 6205\\
Laboratoire de Math\'ematiques de Bretagne Atlantique\\
%6, Avenue Le Gorgeu, C.S. 93837, 29238 Brest Cedex 3, France.\\
e-mail : luisgall@univ-brest.fr - rahavand@univ-brest.fr}
\maketitle
\begin{itemize}
\item[a)]
Running head: Odd prime divisors
\item[b)]
\item[c)]
\item[d)]
Corresponding author:
\begin{center} Olivier Rahavandrainy
\end{center}
\end{itemize}
\newpage~\\
{\bf{Abstract}}
We give \emph{admissible} conditions satisfied by the set of odd prime divisors of perfect polynomials over $\F_2$.
This allows us to prove a new characterization of \emph{all} known
perfect polynomials, and open a way to find more of them (if they exist).

{\section{Introduction}}
Let $A \in \F_2[x]$ be a nonzero polynomial. $A$ is \emph{even} if it has a linear factor and it is
\emph{odd} otherwise. We define a
\emph{Mersenne prime} (polynomial) over $\F_2$ as an irreducible polynomial of the form
$1+x^a(x+1)^b$, for some positive integers $a,b$.
Let $\omega(A)$ (resp. $\sigma(A)$) denote the number of distinct irreducible factors of (resp. the sum of all divisors of) $A$ over $\F_2$
($\sigma$ is a multiplicative
function). If $\sigma(A) = A$, then we
say that $A$ is \emph{perfect}. Finally, we say that a perfect polynomial is \emph{indecomposable} if it does not factor in two coprime nonconstant perfect polynomials.\\
In the rest of the paper, $\N$ ($\N^*$) denotes the set of (positive) natural numbers and we put:
$$\begin{array}{l}
M_1=1+x+x^2,\ M_2=1+x+x^3,\ M_3=\overline{M_2}=1+x^2+x^3,\\
M_4=1+x+x^2+x^3+x^4, M_5=\overline{M_4} = 1+x^3+x^4,\\
M_6=1+x^3+x^5,\ M_7=1+x^3+x^7,\ M_8=1+x^6+x^7,\\
M_9=\overline{M_6},\  M_{10}=\overline{M_7},\  M_{11}=\overline{M_8},\\
M_{12} = x^9+x+1, M_{13}=\overline{M_{12}}=x^9+x^8+1,\\
T_1 =x^2(x+1)M_1, T_2=\overline{T_1}, \\
T_3 = x^4(x+1)^3M_4, T_4 =\overline{T_3},\
T_5 = x^4(x+1)^4M_4\overline{M_4} = \overline{T_5},\\
T_6 = x^6(x+1)^3M_2\overline{M_2}, T_7= \overline{T_6},\\
T_8 = x^4(x+1)^6M_2\overline{M_2} M_4 \text{ and } T_9 = \overline{T_8},\\
T_{10} = x^2(x+1)(x^4+x+1){M_1}^2, \ T_{11} = \overline{T_{10}},
\end{array}$$
where $\overline{S}$ designs the polynomial obtained from a polynomial $S$, by substituting $x$ by $x+1$.\\
The only known nontrivial perfect polynomials are $T_1,\ldots, T_{11}$.
We have characterized (\cite{Gall-Rahav12}, Theorem 1.1) the first nine of them, $T_1, \ldots, T_9$, as the ones which are of the form
$\displaystyle{x^a(x+1)^b \prod_j {P_j}^{2^{n_j}-1}}$, where all the $P_j$'s are Mersenne. Furthermore, $T_{10}$ and $T_{11}$ are the ones of the form $x^a(x+1)^b M^{2h} \sigma(M^{2h})$, where $M$ is Mersenne (\cite{Gall-Rahav13}, Theorem 1.4). We would like to characterize as far as possible those polynomials and by the way, to discover other ones. We remark that odd prime polynomials involving in those polynomials are all Mersenne, except: $S_1:=1+x+x^4 = 1+x(x+1)(1+x+x^2)$.

For an odd polynomial, we introduce by means of Mersenne polynomials, the notions of \emph{representation} and \emph{length} (Sections  \ref{admissible} and \ref{lengthnotion}). We define a \emph{$k$-Mersenne} polynomial as a polynomial with length $k$. A Mersenne polynomial is a \emph{$1$-Mersenne} one. A $2$-Mersenne polynomial is of the form $1+x^a(x+1)^bM^c$, where $M$ is Mersenne and $a,b, c \in \N^*$. In particular, $S_1 = 1+x+x^2$ is the $2$-Mersenne polynomial of lowest degree (it is irreducible).

Even if that kind of polynomials is ``simple'', studying its irreducibility remains difficult. For that reason, in this paper, we start with $2$-Mersenne primes of the form ${M_1}^{abc}$. But, even in this case, we are not able to consider all possible situations because there are too much things to explore.

In order to continue our investigation, we explain in Section \ref{admissible}, how and why we choose the Mersenne's: $M_1,\ldots, M_{13}$ and the following $2$-Mersenne primes: $S_1, \ldots, S_{15}$.\\
For $Q \in \F_2[x]$ odd, we put $Q^{abc}:= 1+x^a(x+1)^bQ^c$ and $\displaystyle{Q^*(x):= x^{\deg(Q)} \cdot Q(\frac{1}{x})}$ (the reciprocal of $Q$). We remark that $\overline{Q^{abc}} = \overline{Q}\ ^{bac}$.
$$\begin{array}{l}
S_1:={M_1}^{111} = \overline{S_1},\ S_2: = {M_1}^{221},\
S_3: = {M_1}^{134}, \ S_4:= {M_1}^{311}, \
S_5: = {M_1}^{131},\\
S_{6}: = {M_1}^{314},\ S_{7}:= {M_1}^{113}, \ S_{8}:= {M_1}^{331},\
S_9: = {M_1}^{115}, \ S_{10}:= {M_1}^{411},\\
S_{11}:= {M_1}^{121}, \ S_{12}:= {M_1}^{212}, \ S_{13}:= {M_1}^{141},
\ S_{14}:= {M_1}^{211}, \ S_{15}:= {M_1}^{122}.
\end{array}$$
We shall prove (Theorem \ref{mainresult1}) that the known perfect polynomials are those which are divisible only by Mersenne primes lying in
${\cal{F}}_1:= \{M_1, \ldots, M_{13}\}$ and by $2$-Mersenne primes in ${\cal{F}}_2:=\{S_1,\ldots,S_{15}\}$. We sketch at the beginning of Section \ref{usefulfacts}, a manner to build perfect polynomials from an ``admissible'' family ${\cal{F}}$ of polynomials. Here, ${\cal{F}} := {\cal{F}}_1 \cup {\cal{F}}_2$ (Lemma \ref{Fisadmissible}).
\begin{theorem} \label{mainresult1}
Let $\displaystyle{A = x^a(x+1)^b \prod_{i=1}^{13} {M_i}^{c_i} \cdot \prod_{j=1}^{15} {S_j}^{d_j}=x^a(x+1)^b \ A_1}$, \\
$a,b, c_i, d_j \in \N, \ a,b \geq 1$ and $A_1 \not= 1$.
Then,
$A$ is (indecomposable) perfect if and only if $A, \overline{A} \in \{T_1, \ldots, T_{11}\}$.
\end{theorem}

\section{Proof of Theorem \ref{mainresult1}}
We prove ``necessity'' since sufficiency is already true.\\
For some odd integers $u,v, u_i, v_j$ and for some $n,m, n_i, m_j \in \N$, put:
\begin{equation} \label{lesexposants}
a =2^nu-1,\ b=2^mv-1,\ c_i = 2^{n_i}u_i-1,\ d_j = 2^{m_j}v_j-1, i \leq 13, j \leq 15.
\end{equation}

In order to avoid several cases on $u,v,u_1,\ldots$, we give (Corollary \ref{specialcases}) an upper bound of each integer $n,m, n_1,\ldots$ involving in $A$. Then, by direct (Maple) computations (in three steps), we get (quickly) Theorem \ref{mainresult1}. In the first step, we dress a list of $[n,u,m,v,n_1,u_1,n_2,u_2]$ such that $a \geq 1, a \leq b$ and $c_2 =\gamma_2$ (see Lemmas \ref{exponentsofsigmA} and \ref{corol1A=sigmA}). We obtain $10944$ such $8$-tuples. In the second step, we apply the conditions: $d_j = \delta_j$ to get $4484$ $18$-tuples of the form $[n,u,m,v,n_1,u_1,n_2,u_2, d_1, \ldots, d_8,m_1,v_1]$. In the third step, we apply the conditions: $a=\alpha$ and $b=\beta$. We get $44$ polynomials. Among  these latter, we find $A$ such that $a\leq b$ and $\sigma(A) + A$ equals $0$. Theorem \ref{mainresult1} follows.

\begin{proposition} \label{allcases}
i) $u \geq 3$ or $v \geq 3$.\\
ii) $u, v \in \{1,3,5,7,9,13,15\}$.\\
iii) $u_1 \in \{1,3,5,7,15\}$, $u_2, u_3 \in \{1,3\}$, $u_i =1$ if $i \geq 4$.\\
iv) $v_1, v_2 \in \{1,3\}$, $v_j = 1$ if $j \geq 3$.
\end{proposition}
\begin{proof}
i): if $u=v=1$, then $x^a(x+1)^b$ is perfect so $A_1$ is odd and perfect. Lemma \ref{oddperfect} implies that $A_1$ is a square, which is impossible because $A_1 \not= 1$.\\
ii):
One has:
$$\sigma(x^a) = (1+x)^{2^n-1} \cdot (\sigma(x^{u-1}))^{2^n}, \ \sigma((x+1)^b) = x^{2^m-1} \cdot (\sigma((x+1)^{v-1}))^{2^m}.$$
Since any odd divisor of $\sigma(x^{u-1})$ and $\sigma((x+1)^{v-1})$ must belong to ${\cal{F}}$, our results about $u$ and $v$ follow from
Lemma \ref{divsigmx2h}.\\
iii): In order to keep in ${\cal{F}}$, all odd prime divisors of $\sigma(M_i^{c_i})$, we take, by Lemma \ref{alldivsigmMers2h},
$u_1 \in \{1,3,5,7,15\}$, $u_2, u_3 \in \{1,3\}$ and $u_i =1$ if $i \geq 4$.\\
iv): Same reason as in iii).
\end{proof}
\begin{lemma} [\cite{Gall-Rahav13}, Lemma 2.4] \label{oddperfect}~\\
Any odd perfect polynomials over $\F_2$ is a square.
\end{lemma}
From Proposition \ref{allcases} and Corollary \ref{finalconditions}, we get
\begin{corollary} \label{specialcases}
The integers $u,v,n,m,n_j, m_j,u_j,v_j$ satisfy:
$$\begin{array}{l}
u,v \in \{1,3,5,7,9,13,15\}, \ u_1 \in \{1,3,5,7,15\}, \ u_2, u_3, v_1 \in \{1,3\},\\
n,m ,n_1 \leq 4, \  n_2, n_3, m_1 \leq 3, \ n_4, n_5 \leq 5, \text{ and for $j \geq 2$, $v_j = 1, \ m_j \in \{0,1\}$}.
\end{array}$$
\end{corollary}

\section{Useful facts} \label{usefulfacts}
We explain how and why we have chosen the above family ${\cal{F}} = {\cal{F}}_1 \cup {\cal{F}}_2$, called \emph{admissible}.
We begin with the reciprocity stability (Sections \ref{Mersreciproque} and \ref{reciproque}) to get first members of ${\cal{F}}$, namely $M_1, \ldots, M_4,M_{12}, M_{13}, S_1, S_{2}, S_{3}, \ldots$
After that, in Section \ref{lesdiviseursdesigmA}, for $S =x, x+1$, Mersenne or $2$-Mersenne, we search possible prime divisors of some $\sigma(S^{2h})$, $h \in \N^*$, to get other members of ${\cal{F}}$. By the way, we find all possible exponents for divisors of $\sigma(A)$.
Finally, Section \ref{sigmAandA} gives all the conditions to obtain: $\sigma(A) = A$.
\subsection{Admissible family} \label{admissible}
Inspired by Lemma \ref{bar-star-stable}, we get Definition \ref{defadmissible} and Corollary \ref{oddadmissible}.
\begin{lemma} \label{bar-star-stable} Let $B$ be an even non splitting perfect polynomial over $\F_2$ and $Q$ an odd prime divisor of $B$. Then:\\
i) there exists $h \in \N^*$ such that $x^{2h}$ or $(x+1)^{2h}$ divides $B$.\\
ii) $1+Q$ divides $B$ or $\sigma(Q^{2h})$ divides $B$, for some $h \in \N^*$.\\
iii) if  $x^{2h}$ divides $B$, then for any prime factor $P$ of $\sigma(x^{2h})$, $P^*$  also divides $\sigma(x^{2h})$ and $\overline{P}$ divides $\sigma((x+1)^{2h})$.
\end{lemma}
\begin{proof}
i): $B$ does not split, so the exponent of $x$ (resp. of $x+1$) in $B$ is of the form $2^nu-1$ (resp. $2^mv-1$), where $u, v$ is odd, $u \geq 3$ or $v \geq 3$.\\
ii): The exponent  of $Q$ in $B$ is of the form $2^ts-1$, with $s$ odd and $t \geq 1$. If $s=1$, then $1+Q$ divides $(1+Q)^{2^t-1} = \sigma(Q^{2^t-1})$ which in turn, divides $\sigma(B)=B$. If $s\geq 3$, then $\sigma(Q^{s-1})$ divides $\sigma(Q^{2^ts-1})$ and $B$.\\
iii): We remark that $(\sigma(x^{2h}))^* = \sigma(x^{2h})$. So, $\displaystyle{\sigma(x^{2h}) = \prod_i U_i \cdot \prod_j V_j {V_j}^*}$, where
$U_i = {U_i}^*$ and $V_j \not= {V_j}^*$. Our result follows.
\end{proof}
\begin{definition} \label{defadmissible}
\emph{A family ${\cal{G}}$ of odd irreducible polynomials is} admissible \emph{if it satisfies at least i), ii) or iii):\\
i) For any $T \in {\cal{G}}$,
$T^* \in {\cal{G}}$ or $\overline{T} \in {\cal{G}}$.\\
ii) There exists $h \in \N^*$ such that $\sigma(x^{2h})$ or $\sigma((x+1)^{2h})$  factors in ${\cal{G}}$.\\
iii) For any $T \in {\cal{G}}$, $1+T$ or $\sigma(T^{2h})$ factors in  ${\cal{G}} \cup \{x,x+1\}$, for some $h \in \N^*$.}
\end{definition}
\begin{corollary} \label{oddadmissible}
The set of odd prime divisor(s) of any even non splitting perfect polynomials is admissible.
\end{corollary}
\begin{remark}
\emph{An admissible family is not necessarily stable both under $Q \mapsto \overline{Q}$ and $Q \mapsto Q^*$.
For example, ${\cal{G}} = \{M_1, \ldots, M_5\}$ is admissible giving the first nine perfect polynomials $T_1, \ldots, T_9$. However, ${M_5}^* = S_1 \not\in {\cal{G}}$.}
\end{remark}
\begin{lemma} \label{Fisadmissible}
The family ${\cal{F}} = {\cal{F}}_1 \cup {\cal{F}}_2$ is admissible.
\end{lemma}
%\begin{proof} i) and iii): For any $T \in {\cal{F}}$, by direct computations: $1+T$ factors in ${\cal{F}} \cup \{x,x+1\}$ and $(\overline{T} \in %{\cal{F}}$ or $T^* \in {\cal{F}})$. We also see (by the way) that ${M_i}^*, {S_j}^* \not\in {\cal{F}}$ if $i \in \{9,10,11\}$ and $j \in %\{7,8,11,12,13\}$. We get ii) from Lemma \ref{divsigmx2h}.
%\end{proof}
\subsection{Representation and length of an odd polynomial} \label{lengthnotion}
\begin{definition} \label{representation}
{\emph{Let $P$ be an odd polynomial over $\F_2$. \\
i) A representation of $P$ is the sequence}} $repr(P):=[[a_1,b_1], \ldots, [a_r,b_r]],$\\
{\emph{where $P_1=P$ and for $j \leq r-1$:
$$a_j = val_x(1+P_j),\ b_j = val_{x+1}(1+P_j),\ P_{j+1} = \frac{1+P_j}{x^{a_j}(x+1)^{b_j}},$$
where $val_x(S)$ $($resp. $val_{x+1}(S))$ denotes the valuation of $S$, at $x$ $($resp. at $x+1)$.
In this case, $\displaystyle{\deg(P) = \sum_{j=1}^r (a_j + b_j)}$.}}\\
{\emph{ii) The length of $P$, denoted by $length(P)$, is the ``length'' of the representation of $P$ defined above.}}
\end{definition}
\begin{lemma}
i) $P$ is Mersenne if and only if $length(P) = 1$.\\
ii) $length(\overline{P}) = length(P)$.
\end{lemma}
\begin{definition} \label{2-mersenne}
\emph{An odd polynomial $P$ is \emph{$k$-Mersenne} if it is of length $k$.}
\end{definition}
\begin{remarks}
\emph{i) An odd polynomial $P$ is {\emph{$2$-Mersenne}} if and only if it is of the form $1+x^a(x+1)^bM^c$, where $M$ is Mersenne and $a,b,c \in \N^*$.}\\
\emph{ii) The length of $P$ and that of its reciprocal $P^*$ are distinct, in general. For example, $length(S_1) = length(1+x+x^4)= 2$, but $length({S_1}^*) = length(1+x^3+x^4) = 1$.}\\
\emph{iii) If $P$ is $2$-Mersenne, then $P^*$ may be not $2$-Mersenne (example: $P=S_1$).}
\end{remarks}

\subsection{Mersenne primes with reciprocal Mersenne} \label{Mersreciproque}
\begin{lemma} [see \cite{Canaday}, p. 728-729]~\\
Let $M$ be a Mersenne prime such that $M^*$ is also Mersenne. Then\\
i) $M \in \{M_1, M_4\}$ if $M=M^*$.\\
ii)  $M \in \{M_2, M_3,M_{12}, M_{13}\}$ if $M \not= M^*$.
\end{lemma}
\subsection{Reciprocal of some $2$-Mersenne polynomials} \label{reciproque}
In this section, we suppose that $Q=M_1=1+x+x^2$ and $Q^{abc}$ is irreducible. So,
$(Q^{abc})^* = x^{a+b+2c}+(x+1)^b(x^2+x+1)^c, \text{ with $\gcd(a,b,c) =1$}.$\\
We study the case where $(Q^{abc})^*$ is Mersenne or  $(Q^{abc})^* = Q^{abc}$ or $(Q^{abc})^*$ is an other $2$-Mersenne. We would like to precise that we are not able to get all such polynomials, because of the difficulty to prove their irreducibility.

\subsubsection{Case  $(Q^{abc})^*$ is Mersenne}
We put $(Q^{abc})^* = 1+x^d(x+1)^e$ for some positive integers $d,e$. One has: $a+b+2c = d+e$.
We consider several cases on the parity of $a,b,c,d$ and $e$. By the following lemmas, we obtain:
$$S_1,S_{10}, S_{14}, S_{15}, \text{ with $(S_1)^*=M_5, (S_{10})^* = M_7$, $(S_{14})^* = M_6, (S_{15})^* = M_8.$}$$

\begin{lemma}
$a$ and $b$ are not both even.
\end{lemma}
\begin{proof}
If $a$ and $b$ are both even, then $c$, $d$ and $e$ are all odd. By differentiating, we get:
$(x+1)^b(x^2+x+1)^{c-1} = x^{d-1} (x+1)^{e-1}$ and thus $c=1$, $d=1$, $b=e-1$,
$a+e-1+2 = a+b+2c=d+e = 1+e$ and $a=0$, which is impossible.
\end{proof}
\begin{lemma}
If $a$ and $b$ are both odd, then $a=c=1, d=3, b=e=2^r-3$ for some $r \geq 2$. In particular, for $r=2$, we get $S_1$ and $M_5 = (S_1)^*$.
\end{lemma}
\begin{proof}
In this case, $d$ and $e$ are also both odd.\\
- If $c$ is even, then by differentiating, we get:
$(x+1)^{b-1}(x^2+x+1)^{c} = x^{d-1} (x+1)^{e-1}$, which is impossible because $c >0$.\\
- If $c$ is odd, then by differentiating, we get:
$x^2(x+1)^{b-1}(x^2+x+1)^{c-1} = x^{d-1} (x+1)^{e-1}$, $c=1$, $d=3$, $b=e$,
$a+b+2 = a+b+2c=d+e = 3+b,$ so that $a=1$.
Therefore $b+3=a+b+2c$ and
$$1+x^{b+3}= (x+1)^b(x^2+x+1)^c +x^d(x+1)^e = (x+1)^b[(x^2+x+1) +x^3] =(x+1)^{b+3}.$$
Thus, $b+3 = 2^r$, with $r \in \N$. If $r=2$, we get $a=b=c=e=1,d=3$.
\end{proof}
\begin{lemma}
 If $a$ is even and $b$ odd, then $b=c=1, \ d=3, \ a=e=2^r$, with $r \in \N$.
In particular, for $r \in \{1,2\}$, we get $S_{10} = {M_7}^*$ and $S_{14}={M_6}^*$.
\end{lemma}
\begin{proof}
The integer $d+e$ must be odd.\\
- If $c$ is even and $d$ even, then $e$ is odd. By differentiating, we get:
$$x^{a+b+2c-1}+(x+1)^{b-1}(x^2+x+1)^{c} = x^{d} (x+1)^{e-1},$$ which is impossible (take $x=0$).\\
- If $c$ is even and $d$ odd, then $e$ is even. By differentiating, we get:
$$x^{a+b+2c-1}+(x+1)^{b-1}(x^2+x+1)^{c} = x^{d-1} (x+1)^{e},$$ and so $b=1$. Hence,
$x^{a+2c}+(x^2+x+1)^{c} = x^{d-1} (x+1)^{e}$, $d=1$ and $e = a+2c$ even. It follows that:
$$\left\{\begin{array}{l}
x^{a+2c+1}+(x+1)(x^2+x+1)^c = (Q^{abc})^* = 1+x^d(x+1)^e = 1+x(x+1)^e,\\
x^{a+2c}+(x^2+x+1)^{c} =  x^{d-1} (x+1)^{e} = (x+1)^{e}.
\end{array}
\right.$$
Thus,
$$1+x^{a+2c+1}+(x+1)(x^2+x+1)^c = x(x+1)^{e} = x(x^{a+2c}+(x^2+x+1)^{c}).$$
We get the contradiction: $1+(x^2+x+1)^c = 0$ with $c>0$ even.\\
- If $c$ is odd and $d$ even, then $e$ is odd. By differentiating, we get:
$$x^{a+b+2c-1}+x^2(x+1)^{b-1}(x^2+x+1)^{c-1} = x^{d} (x+1)^{e-1}.$$
Hence, $d=2$ and $x^{a+b+2c-3}+(x+1)^{b-1}(x^2+x+1)^{c-1} = (x+1)^{e-1}.$\\
It follows that:
$$\left\{\begin{array}{l}
x^{a+b+2c}+(x+1)^b(x^2+x+1)^c = (Q^{abc})^* = 1+x^d(x+1)^e = 1+x^2(x+1)^e,\\
x^{a+b+2c-3}+(x+1)^{b-1}(x^2+x+1)^{c-1} = (x+1)^{e-1}.
\end{array}
\right.$$
$1+x^{a+b+2c}+(x+1)^b(x^2+x+1)^c =x^2(x+1)^e = (x^3+x^2)(x+1)^{e-1} = (x^3+x^2) [x^{a+b+2c-3}+(x+1)^{b-1}(x^2+x+1)^{c-1}]$.\\
We get: $1+x^{a+b+2c-1} = (x+1)^{b+1}(x^2+x+1)^{c-1}$ and so $a+b+2c-1 = b+1+2c-2$ and $a=0$, which is impossible.\\
- If $c$ is odd and $d$ odd, then $e$ is even. By differentiating, we get:
$$x^{a+b+2c-1}+x^2(x+1)^{b-1}(x^2+x+1)^{c-1} = x^{d-1} (x+1)^{e},$$
Hence, $d=3$ and $x^{a+b+2c-3}+(x+1)^{b-1}(x^2+x+1)^{c-1} = (x+1)^{e}.$\\
It follows that $b=1$. Hence,
$x^{a+2c-2}+(x^2+x+1)^{c-1} = (x+1)^{e}$. From:
$$\left\{\begin{array}{l}
x^{a+2c+1}+(x+1)(x^2+x+1)^c = (Q^{abc})^* = 1+x^d(x+1)^e = 1+x^3(x+1)^e,\\
x^{a+2c-2}+(x^2+x+1)^{c-1} = (x+1)^{e},
\end{array}
\right.$$
Thus,
$$1+x^{a+2c+1}+(x+1)(x^2+x+1)^c = x^3(x+1)^{e} = x^3[x^{a+2c-2}+(x^2+x+1)^{c-1}].$$
So, $1+(x+1)(x^2+x+1)^c = x^3(x^2+x+1)^{c-1}.$ Therefore,
$$0 = 1+(x+1)(x^2+x+1)^c + x^3(x^2+x+1)^{c-1} = 1+(x^2+x+1)^{c-1}.$$
We get $c=1$, $a=e$ and $1+x^a = (x+1)^e$. So, $a=e=2^r$, with $r \geq 1$. For $r \in \{1,2\}$,
we obtain: $(a=e=2, b=c=1,d=3)$ or $(a=e=4, b=c=1,d=3)$.
\end{proof}
\begin{lemma}
If $a$ is odd and $b$ even, then $a=e=1, b=c=2^m, d=3 \cdot 2^m$, with $m \geq 1$.
In particular, for $m=1$, we get $S_{15} = {M_8}^*$.
\end{lemma}
\begin{proof}
The integer $d+e$ is odd.\\
- If $c$ is even, then by differentiating, we get: $x^{a+b+2c-1} = x^{\alpha} (x+1)^{\beta}$. Hence $\beta = 0$ so that $e=1$ and
$\alpha=d=2^m u$ is even. Thus,
$$x^{a+b+2c}+(x+1)^b(x^2+x+1)^c = 1+x^d(x+1)^{e} = 1+x^{d+1} + x^d.$$
So, $(x+1)^b(x^2+x+1)^c =1+x^d = (1+x)^{2^m} (1+x+\cdots+x^{u-1})^{2^m}$, $u=3, b=c=2^m, d=3 \cdot 2^m$, with $m \geq 1$.
\end{proof}
\subsubsection{Case  $(Q^{abc})^* =  Q^{abc}$}
\begin{proposition} \label{Qstar=Q}
If $(Q^{abc})^* = Q^{abc}$ then
 $(a=1, b = 2^n-1, c=2^n)$ or  $(a=3, b=c=2^n-1)$.
In particular, for $n=2$ (resp. $n=1$), we get $S_3={M_1}^{134}$ (resp. $S_4={M_1}^{311}$).
\end{proposition}
\begin{proof}
One has: \begin{equation} \label{form1qstarq}
1+x^{a+b+2c} =(x+1)^b(x^2+x+1)^c (1+x^a).
\end{equation}
Put $a+b+2c =2^n u$ and $a=2^m v$, with $n, m \geq 0$ and $u,v$ odd.\\
Equality (\ref{form1qstarq}) gives:
\begin{equation} \label{form2qstarq}
(1+x)^{2^n} (1+x+\cdots+x^{u-1})^{2^n} =(x+1)^{b+2^m}(x^2+x+1)^c (1+x+\cdots+x^{v-1})^{2^m}.
\end{equation}
Thus, $2^n=b+2^m$, $n>m$, $x^2+x+1$ divides $1+x+\cdots+x^{u-1}$ so that $3$ divides $u$ and
\begin{equation} \label{form3qstarq}
(1+x+\cdots+x^{u-1})^{2^n} =(x^2+x+1)^c (1+x+\cdots+x^{v-1})^{2^m}.
\end{equation}
- If $v=1$, then $u=3$, $c=2^n, a=2^m$. If $m\geq 1$, then $a,b$ and $c$ are all even. it contradicts the fact that $\gcd(a,b,c)=1$.
So, $m=0$, $a=1, b = 2^n-1, c=2^n$. For $n=2$, we obtain ${M_1}^{134} = S_3$.\\
- If $v=3$, then $u=3$, $b=c=2^n-2^m$. As above, we must have $m=0$, $b=c=2^n-1$, $a=3\cdot 2^m = 3$. For $n=1$, we obtain
${M_1}^{311}=S_4$.\\
- If $v>3$, since $1+x+\cdots+x^{w}$ is square-free for any even integer $w$, Equation (\ref{form3qstarq}) implies:
$$2^n=c=2^m,\ 1+x+\cdots+x^{u-1} = (x^2+x+1) (1+x+\cdots+x^{v-1}),$$
which is impossible because $n\not=m$.
\end{proof}
\subsubsection{Case  $(Q^{abc})^* =  Q^{def} \not= Q^{abc}$}
One has:
\begin{equation} \label{formule3}
1+x^{a+b+2c}+(x+1)^b(x^2+x+1)^c = x^{d}(x+1)^e(x^2+x+1)^f.
\end{equation}
\begin{lemma}
If $a+b+2c = 2^n u$ for some $u, n \in \N$ with $u$ odd, then $u \equiv 0 \mod 3$.
\end{lemma}
\begin{proof}
In this case, Equality (\ref{formule3}) implies that $1+x+x^2$ divides $1+x^{2^nu}$. Thus, $3$ divides $2^n u$.
\end{proof}
In order to keep the paper no more long, we only consider the case where $u=3$, so that $a+b+2c = 3 \cdot 2^n$ and Equality (\ref{formule3}) becomes:
\begin{equation} \label{formule3-2}
(1+x)^{2^n} \cdot (x^2+x+1)^{2^n} +(x+1)^b(x^2+x+1)^c = x^{d}(x+1)^e(x^2+x+1)^f.
\end{equation}
By comparing $b$ and $2^n$, we get Lemmas \ref{b=2n}, \ref{b<2n} and \ref{b>2n} which, in turn, give other $2$-Mersenne primes, namely
$S_2, S_5 = {S_2}^*, S_6$ and $S_9 = {S_6}^*$.
We need the following facts.
\begin{lemma} \label{petitlemme}
i) $1+(x^2+x+1)^a = x^b(x+1)^c$ if and only if $a=b=c=2^r,$ for some $r \in \N$.\\
ii) $(x+1)^a+(x^2+x+1)^b = x^c$ if and only if $(a=b=2^r, c= 2b)$ or $(b=c=2^r,a=2b)$ or $(b=2^r, a=c=3b)$ for some $r \in \N$.\\
iii) $(x+1)^a(x^2+x+1)^b = 1+x^c$ if and only if $(a=b=2^r, c= 3b)$ for some $r \in \N$.\\
iv) $(x+1)^a+(x+1)^b = x^c(x+1)^d$, with $a \leq b$ if and only if $(b=a+2^r,c=2^r, d= a)$ for some $r \in \N$.\\
v) $1+(x+1)^a = x^b(x^2+x+1)^c$ if and only if $(a=b=2^r, c=0)$ or $(a = 3 \cdot 2^r, b=c=2^r)$.
\end{lemma}
\begin{proof}
By direct computations by putting: $a =2^n u$, $b=2^mv,\ldots$ and by differentiating, if necessary.
\end{proof}
\begin{lemma} \label{b=2n}
If $b=2^n$, then $(a=2, b=2^n, f=c=2^n-1, d=1, e=2^n+1)$. For $n=1$, we obtain
$S_2={M_1}^{221}$ and  $S_5 = {M_1}^{131} = (S_2)^*$.
\end{lemma}
\begin{proof}
Equality (\ref{formule3-2}) implies:
\begin{equation} \label{formule3-3}
(1+x)^{2^n} [(x^2+x+1)^{2^n} +(x^2+x+1)^c] = x^{d}(x+1)^e(x^2+x+1)^f.
\end{equation}
Thus, $c \not= 2^n$ and $a+2c = 2 \cdot 2^{n}$.\\
- If $c< 2^n$, then $f=c$ and $(x+1)^{2^n} [1+(x^2+x+1)^{2^n-c}]=x^d(x+1)^e$. So, $1+(x^2+x+1)^{c-2^n}$ splits.
Lemma \ref{petitlemme}-i) implies that $2^n-c = 2^r$ and thus $x^d(x+1)^e =(x+1)^{2^n} [1+(x^2+x+1)^{2^n-c}] = x^{2^r}(x+1)^{2^n+2^r}$.\\
It follows that $(b=2^n, f=c=2^n-2^r, a=2 \cdot 2^r, d=2^r, e=2^n+2^r$. Since $\gcd(a,b,c)=1$, we get $r=0$ and
$a=2, b=2^n, f=c=2^n-1, d=1, e=2^n+1$.\\
- If $c > 2^n$, then $f=2^n$ and $(x+1)^{2^n} [1+(x^2+x+1)^{c-2^n}]=x^d(x+1)^e$. So, $1+(x^2+x+1)^{c-2^n}$ splits.
Lemma \ref{petitlemme}-i) implies that $c-2^n = 2^r$. We get the contradiction: $a = 3\cdot 2^n -b-2c = - 2 \cdot 2^r \leq 0$.
\end{proof}
\begin{lemma} \label{b<2n}
If $b<2^n$, then
$(a=3, e=b=2^n-3, c= 2^n, d=1,f=2^n+1)$ or $(a=1, e=b=2^n-3, c=2^n+1, d=3, f=2^n)$. For $n=2$, we obtain
$S_6={M_1}^{314}$ and  $S_9 = {M_1}^{115} = (S_6)^*$.
\end{lemma}
\begin{proof}
We consider three cases.\\
- If $c=2^n$, then $a+b = 2^n$.
Equality (\ref{formule3-2}) gives:
\begin{equation} \label{formule3-4}
(1+x)^{b} (x^2+x+1)^{2^n} [1+(x+1)^{2^n-b}] = x^{d}(x+1)^e(x^2+x+1)^f.
\end{equation}
Thus, $1+(x+1)^{2^n-b}$ is of the form $x^{t_1}(x^2+x+1)^{t_2}$. Lemma \ref{petitlemme}-v) implies that either
$(a=2^n-b = d=2^r, e=b,f=2^n=c, n>r)$ or $(a=2^n-b = 3 \cdot 2^r, d=2^r, e=b, f=2^n + 2^r)$.\\
The first case does not happen because $(Q^{def}) \not= Q^{abc}$.\\
Since $\gcd(a,b,c) = 1$ in the second case, we must have: $r=0$ and so $a=3, e=b=2^n-3, c= 2^n, d=1,f=2^n+1$.\\
- If $c < 2^n$, then Equality (\ref{formule3-2}) gives:
\begin{equation} \label{formule3-5}
(1+x)^{b} (x^2+x+1)^{c} [(x+1)^{2^n-b}(x^2+x+1)^{2^n-c} + 1] = x^{d}(x+1)^e(x^2+x+1)^f.
\end{equation}
Thus, $e=b, f=c$ and $(x+1)^{2^n-b}(x^2+x+1)^{2^n-c} + 1 = x^d$.
Lemma \ref{petitlemme}-iii) implies that $2^n-b =2^n-c= 2^r, d=3\cdot 2^r$ and thus $e=f=b=c=2^n-2^r, a=3\cdot 2^r=d$, which is
impossible because $(Q^{def}) \not= Q^{abc}$.\\
- If $c > 2^n$, then Equality (\ref{formule3-2}) gives:
\begin{equation} \label{formule3-6}
(1+x)^{b} (x^2+x+1)^{2^n} [(x+1)^{2^n-b}+(x^2+x+1)^{c-2^n-c}] = x^{d}(x+1)^e(x^2+x+1)^f.
\end{equation}
Thus, $e=b, f=2^n$ and $(x+1)^{2^n-b}+(x^2+x+1)^{c-2^n} = x^{d}$. Lemma \ref{petitlemme}-ii) implies that
$(2^n-b=c-2^n=2^r, d=2\cdot 2^r)$ or $(2^n-b=2\cdot 2^r, c-2^n=d=2^r)$ or $(c-2^n=2^r, 2^n-b=d=3\cdot 2^r)$.
The two first cases give the contradiction: $a = 3 \cdot 2^n - b-2c  \leq 0$.
The third implies that $a=2^r, b=2^n-3 \cdot 2^r, c= 2^n+2^r$. Hence, $r=0$ because $\gcd(a,b,c)=1$. We get
$a=1, e=b=2^n-3, c=2^n+1, d=3, f=2^n$.
\end{proof}
\begin{lemma} \label{b>2n}
If $b>2^n$, then $(a=1, e=2^n, b=2^n+1, f=c=2^n-1, d=2)$. For $n=1$, we retrieve $S_5$ and $S_2=(S_5)^*$.
\end{lemma}
\begin{proof}
In this case, $3 \cdot 2^n = a+b+2c > 2c +2^n$. So, $2c < 2 \cdot 2^n$ and $c<2^n$.
Equality (\ref{formule3-2}) implies:
\begin{equation} \label{formule3-4}
(1+x)^{2^n} (x^2+x+1)^{c} \cdot [(x^2+x+1)^{2^n-c} +(x+1)^{b-2^n}] = x^{d}(x+1)^e(x^2+x+1)^f.
\end{equation}
Thus, $e=2^{n}, f=c$ and $(x+1)^{b-2^n} + (x^2+x+1)^{2^n-c} = x^{d}$. Lemma \ref{petitlemme}-ii) implies that
$(b-2^n=2^n-c=2^r, d=2\cdot 2^r)$ or $(b-2^n=2\cdot 2^r, 2^n-c=d=2^r)$ or $(2^n-c=2^r, b-2^n=d=3 \cdot 2^r)$.\\
It follows that $(a=2^r, b=2^n+2^r,c=2^n-2^r)$ or $(a \leq 0)$. As above, we get $r=0$ and
$a=1, e=2^n, b=2^n+1, f=c=2^n-1, d=2$.
\end{proof}
\begin{remark}
For any $T \in {\cal{F}}$, one has $\overline{T} \in {\cal{F}}$ but
$$\text{$T^* \not\in {\cal{F}}$ if $T \in \{M_9,M_{10}, M_{11}, S_7, S_8, S_{11},S_{12},S_{13}\}$}.$$
\end{remark}

\subsection{Prime divisors of $\sigma(A)$ and their suitable exponents} \label{lesdiviseursdesigmA}
In order to compare $A$ and $\sigma(A)$, we give in this section, all divisors of the latter and their suitable exponents. With the same notations as in (\ref{lesexposants}), we may write:
\begin{equation} \label{sigmaAdetails}
\begin{array}{l}
\displaystyle{\sigma(A) = \sigma(x^a) \sigma((x+1)^b)) \prod_{i=1}^{13} \sigma({M_i}^{c_i}) \prod_{j=1}^{15} \sigma({S_j}^{d_j})},\\
\sigma(x^a) = (x+1)^{2^n-1} \cdot [\sigma(x^{u-1})]^{2^n},\
\sigma((x+1)^b) = x^{2^m-1} \cdot [\sigma((x+1)^{u-1})]^{2^m},\\
\sigma({M_i}^{c_i}) = (1+M_i)^{2^{n_i}-1} \cdot [\sigma({M_i}^{u_i-1})]^{2^{n_i}},\
\sigma({S_j}^{d_j}) = (1+S_j)^{2^{m_j}-1} \cdot [\sigma({S_j}^{v_j-1})]^{2^{m_j}}.
\end{array}
\end{equation}
We need then to know all $h \in \N^*$ such that $\sigma(S^{2h})$ factors in ${\cal{F}}$, for $S \in \{x,x+1\} \cup {\cal{F}}$. In this case, we shall put
\begin{equation} \label{expressionsigmA}
\displaystyle{\sigma(A) = x^{\alpha} (x+1)^{\beta} \prod_{i=1}^{13} {M_i}^{\gamma_i} \prod_{j=1}^{15} {S_j}^{\delta_j}},
\text{ where $\alpha, \beta, \gamma_i,\ \delta_j \in \N$}.
\end{equation}
\begin{lemma} \label{S2h-squarefree}
For any $h \in \N^*$ and for any $S \in \{x,x+1\} \cup {\cal{F}}$, $\sigma(S^{2h})$ is odd and square-free.
\end{lemma}
\begin{proof}
The oddness is obvious. It is well-known that $\sigma(S^{2h})$ is square-free if $S \in \{x,x+1\} \cup {\cal{F}}_1$. Now, consider $S = {M_1}^{abc}=1+x^a (x+1)^b {M_1}^c \in {\cal{F}}_2$. Put $T = \sigma({S}^{2h}) = (1+S)(1+S +\cdots + {S}^{h-1})^2 + {S}^{2h}$. One has $T' = S' \cdot (1+S +\cdots + {S}^{h-1})^2$. We claim that $\gcd(T,T') = 1$. Let $D$ be a common prime divisor of $T$ and $T'$. If $D$ divides $1+S +\cdots + {S}^{h-1}$, then $D$ divides ${S}^{2h}$ and hence $D=1$.
If $D$ divides $S'$, then by direct computations, $D \in \{1, M_1\}$ because $D$ is odd. Thus $D=1$, by Lemma \ref{divisorofTT'}.
\end{proof}
\begin{lemma} \label{divisorofTT'}
For any $S \in {\cal{F}}_2$, $M_1$ does not divide $\sigma({S}^{2h})$.
\end{lemma}
\begin{proof}
Put $S =1+x^c (x+1)^d {M_1}^e$. If $\alpha$ is a root of $M_1$, then $1 =1+0 = 1+\alpha^c (\alpha+1)^d (M_1(\alpha))^e= S(\alpha)$ and so $(\sigma({S}^{2h}))(\alpha) = 1+S(\alpha) + \cdots + (S(\alpha))^{2h} =  1 \not= 0$.
\end{proof}
Direct computations give
\begin{lemma} \label{totaldegreeF}
One has: $\displaystyle{\sum_{D \in {\cal{F}}} \deg(D) = 184}$.
\end{lemma}

\begin{lemma} \label{divsigmx2h}
If $\sigma(x^{2h})$ and $\sigma((x+1)^{2h})$ factor in ${\cal{F}}$, then $h \in \{1,2,3,4,6,7\}$.
In this case,
$$\begin{array}{l}
\sigma(x^{2}) = \sigma((x+1)^2) = M_1,\ \sigma(x^{4}) = M_4, \ \sigma((x+1)^{4}) = M_5, \\
\sigma(x^{6}) = \sigma((x+1)^{6}) =  M_2 M_3,\
\sigma(x^{8}) = M_1S_4, \ \sigma((x+1)^{8}) = M_1S_5, \\
\sigma(x^{12}) = S_3,\ \sigma((x+1)^{12})=S_6,\
\sigma(x^{14}) = \sigma((x+1)^{14}) = M_1M_4M_5S_1.
\end{array}$$
\end{lemma}
\begin{proof}
We remark that $\sigma((x+1)^{2h}) = \overline{\sigma(x^{2h})}$. So, it suffices to consider $X_h:=\sigma(x^{2h})$. One has: $\displaystyle{X_h = \prod_{P \in {\cal{F}}} P^{c_P}}$, where $c_P \in \{0,1\}$, because $X_h$ is square-free. Moreover, $2h = \deg(x^{2h}) \leq 184$, by Lemma \ref{totaldegreeF}. We get then our result, by direct (Maple) computations (which are done for $h \leq 92$).
\end{proof}

\begin{lemma} \label{alldivsigmMers2h}
Let $M \in {\cal{F}}_1$ such that $\sigma(M^{2h})$ factors in ${\cal{F}}$. Then
$(M=M_1$ and $h \in \{1,2,3,7\})$ or $(M \in \{M_2, M_3\}$ and $h=1)$. We get $\sigma({M_2}^2) = M_1M_5,\ \sigma({M_3}^2) = M_1M_4$ and
$\sigma({M_1}^2) = S_1,\ \sigma({M_1}^4) = S_{8},\ \sigma({M_1}^6) = M_2M_3S_2,\ \sigma({M_1}^{14}) = M_4M_5S_1S_{7}S_{8}.$
\end{lemma}
\begin{proof} As above, we may write
$\displaystyle{\sigma(M^{2h}) = \prod_{P \in {\cal{F}}} P^{c_P}}$, with $c_P \in \{0,1\}$ and $4h \leq 2h\deg(M) \leq 184$. So, $h \leq 46$.
Our results follow by direct computations (which needed about 30 mn).
\end{proof}

\begin{lemma} \label{divsigmS2h}
Let $S \in {\cal{F}}_2$ such that $\sigma({S}^{2h})$ factors in ${\cal{F}}$, then $h=1$ and $S \in \{S_1,S_2\}$.
We get $\sigma({S_1}^{2}) = M_4M_5$ and $\sigma({S_2}^{2}) = S_1S_{7}$.
\end{lemma}
\begin{proof}
Analogous proof: here, $8h \leq 2h \deg(S) \leq 184$. So, $h \leq 23$ (computations took 125 s).
\end{proof}
From Lemmas \ref{divsigmx2h}, \ref{alldivsigmMers2h} and \ref{divsigmS2h}, we get
\begin{corollary} \label{theoddivisors}
i) If $M_i$ and $S_j$ divide $\sigma(A)$, then $i \leq 5$ and $j \leq 8$.\\
ii)  For any $j \in \{2,\ldots,6\}$, ${S_j}^2$ does not divide $\sigma(A)$.
\end{corollary}
\begin{proof}
i): For any $i \geq 6$ and $j \geq 9$, neither $M_i$ nor $S_j$ divides $\sigma(A)$.\\
ii): $S_2, S_3, S_4, S_5, S_6$ respectively divide only $\sigma({M_1}^6), \sigma(x^{12}), \sigma(x^8), \sigma((x+1)^8)$ and $\sigma((x+1)^{12})$. So, for any $j \in \{2,\ldots,6\}$, ${S_j}^2$ does not divide $\sigma(A)$.
\end{proof}
For $w \in \N^*$, we denote by $\chi_{w}$ the indicator function of the singleton $\{w\}$:
$$\text{$\chi_{w}(w) = 1, \chi_{w}(t) = 0$ if $ t\not= w$.}$$
According to notations in (\ref{lesexposants}) and in (\ref{expressionsigmA}), put:
$$\begin{array}{l}
\text{$M_i = 1+x^{a_i} (x+1)^{b_i} \in {\cal{F}}_1$, $S_j = 1+x^{\alpha_j} (x+1)^{\beta_j} {M_1}^{\nu_j} \in {\cal{F}}_2$},\\
\text{$\xi_1 = \chi_{3}(u) + \chi_{9}(u) + \chi_{15}(u), \ \xi_2 = \chi_{3}(v) + \chi_{9}(v) + \chi_{15}(v)$},\\
\text{$\xi_3 = \chi_5(u) + \chi_{15}(u), \ \xi_4 = \chi_5(v) + \chi_{15}(v)$}.
\end{array}
$$
From equalities in (\ref{sigmaAdetails}), one directly has:
\begin{lemma} \label{exponentsofsigmA}
The integers $\alpha, \beta, {\gamma_i}$'s and ${\delta_j}$'s satisfy:
$$\begin{array}{l}
\displaystyle{\alpha = 2^m-1 + \sum_{i=1}^5 (2^{n_i}-1)a_i + \sum_{j=1}^8 (2^{m_j}-1)\alpha_j,}\\
\displaystyle{\beta = 2^n-1 + \sum_{i=1}^5 (2^{n_i}-1)b_i + \sum_{j=1}^8 (2^{m_j}-1)\beta_j,}\\
\displaystyle{\gamma_1 = \sum_{j=1}^8 (2^{m_j}-1)\nu_j + \xi_1 \cdot 2^n + \xi_2 \cdot 2^m + \chi_{3}(u_2) \cdot 2^{n_2}
+\chi_{3}(u_3) \cdot 2^{n_3}},\\
\displaystyle{\gamma_2 = \gamma_3=\chi_7(u) \cdot 2^n + \chi_7(v) \cdot 2^m + \chi_7(u_1) \cdot 2^{n_1},}\\
\displaystyle{\gamma_4 = \xi_3 \cdot 2^n + \chi_{15}(v) \cdot 2^m + \chi_{15}(u_1) \cdot 2^{n_1}+
\chi_{3}(u_3) \cdot 2^{n_3} + \chi_{3}(v_1) \cdot 2^{m_1}},\\
\displaystyle{\gamma_5 = \chi_{15}(u) \cdot 2^n + \xi_4 \cdot 2^m + \chi_{15}(u_1) \cdot 2^{n_1}+
\chi_{3}(u_2) \cdot 2^{n_2} + \chi_{3}(v_1) \cdot 2^{m_1}},\\
\displaystyle{\delta_1 = \chi_{15}(u) \cdot 2^n + \chi_{15}(v) \cdot 2^m + (\chi_{3}(u_1) + \chi_{15}(u_1)) \cdot 2^{n_1}},\\
\delta_2= \chi_{7}(u_1) \cdot 2^{n_1},\ \delta_3= \chi_{13}(u) \cdot 2^{n},\ \delta_4= \chi_{9}(u) \cdot 2^{n},\ \delta_5= \chi_{9}(v) \cdot 2^{m},\\
\delta_6= \chi_{13}(v) \cdot 2^{m},\ \delta_7 = (\chi_{5}(u_1) + \chi_{15}(u_1)) \cdot 2^{n_1}, \ \delta_8 = \chi_{15}(u_1) \cdot 2^{n_1}.
\end{array}$$
\end{lemma}
\subsection{More detailed conditions for $A$ to be perfect} \label{sigmAandA}
We suppose that $A$ is perfect ($A=\sigma(A)$) and we give necessary conditions for the exponents of all prime divisors of $A$.
Those conditions are very useful for computing.
We consider the notations in (\ref{lesexposants}), (\ref{sigmaAdetails}) and (\ref{expressionsigmA}).
\begin{lemma} \label{corol1A=sigmA}
If $A$ is perfect, then:\\
i) $a = 2^n u - 1, b=2^mv-1$ where $n,m \in \N$ and $u, v \in \{1,3,5,7,9,13,15\}$.\\
ii)  $c_i = 0$ and $d_j = 0$, for any $i \geq 6$ and $j \geq 9$.\\
iii) $c_1 =2^{n_1} u_1 -1$ where $u_1 \in \{1,3,5,7,15\}$.\\
iv)  $c_i =2^{n_i} u_i -1$, with $u_i \in \{1,3\}$ if $i \in \{2,3\}$, $u_i = 1$ if $i \in \{4,5\}$.\\
v) $d_j = 2^{m_j} v_j -1$ where $v_1 \in \{1,3\}$, $v_j =1$ if $j \in \{2,\ldots,8\}$.
\end{lemma}
\begin{proof}
They follow from Lemmas \ref{divsigmx2h}, \ref{alldivsigmMers2h} and  \ref{divsigmS2h}.
Remark that $S_2$ divides $\sigma({M_1}^{6})$, $S_7$ divides both $\sigma({M_1}^{14})$ and $\sigma({S_2}^2)$. But, by Corollary \ref{theoddivisors}, ${S_2}^2$ does not divide $\sigma(A)=A$. Hence, $v_2 = v_7 = 1$.
\end{proof}

\begin{lemma} \label{aboutn2345m1}
One has:  $n_2, n_3, m_1 \leq 3$ and $n_4, n_5 \leq 5$.
\end{lemma}

\begin{proof}
We begin with the condition $2^{m_1}v_1-1 = d_1 = \delta_1 = \varepsilon_1 \cdot 2^n +\varepsilon_2 \cdot 2^m + \varepsilon_3 \cdot 2^{n_1}$, where $\varepsilon_k \in \{0,1\}$.
If $m_1 \geq 1$, then $d_1$ is odd. So, $d_1 =1$ or it is of form $2^{h_1} +1$ or $2^{h_1}+ 2^{h_2} +1$, with $h_1, h_2 \geq 1$. Since $v_1 \in \{1,3\}$, we get $m_1 \leq 3$. By analogous proofs, we get $n_2, n_3 \leq 3$.\\
Now, consider $2^{n_4}-1 = c_4 = \gamma_4 = \varepsilon_1 \cdot 2^n +\varepsilon_2 \cdot 2^m + \varepsilon_3 \cdot 2^{n_1}+ \varepsilon_4 \cdot 2^{n_3}+ \varepsilon_5 \cdot 2^{m_1}$, where $\varepsilon_k \in \{0,1\}$ and $m_1, n_3 \leq 3$. If $n_4 \geq 1$, then $c_4$ is odd. So,
$c_4 \in K_1 \cup K_2$, where $K_1 =\{1,3,5,2^{h_1}+1, 2^{h_1}+3, 2^{h_1}+ 2^{h_2} + 1, 2^{h_1}+ 2^{h_2} + 3: h_1, h_2 \geq 1\}$ and
$K_2 =\{2^{h_1}+ 2^{h_2} +2^{h_3} + \ell: \ell \in \{1,3,5,9\}, h_1, h_2, h_3 \geq 1\}$. We get (by Maple computations):
$n_4 \leq 5$. We also have: $n_5 \leq 5$.
\end{proof}
The proof of Lemma \ref{aboutnmn1} needs
\begin{lemma} [see \cite{Gall-Rahav5}] \label{omegaleq4}~\\
If $B$ is an even non splitting perfect polynomial over $\F_2$, with $\omega(B)\leq 4$, then
$B \in \{T_1, T_2, T_3,T_4,T_5, T_6,T_7, T_{10}, T_{11}\}$.
\end{lemma}

\begin{lemma} \label{aboutnmn1}
One has: $n,m ,n_1 \leq 4$.
\end{lemma}
\begin{proof}
We know that $u,v \in \{1,3,5,7,9,13,15\}, u_1 \in \{1,3,5,7,15\}$ and $u_2, u_3, v_1 \in \{1,3\}$, with $u \geq 3$ or $v \geq 3$,
$n_2, n_3, m_1 \leq 3$.\\
- If $u=7$ then from the expression of $\gamma_2 = c_2 = 2^{n_2} u_2 -1$, we get $2^n \leq \gamma_2 =2^{n_2} u_2 -1 \leq 23$. So, $n \leq 4$.\\
- If $u \in \{9,13,15\}$, then  $n=0$ (from the expressions of $\delta_4$, $\delta_3$ and $\gamma_5$).\\
- Analogously, if $v \in \{7,9,13,15\}$, then $m \leq 4$.\\
- If $u_1 = 3$, then $2^{n_1} \leq \delta_1 = d_1 = 2^{m_1} v_1 -1 \leq 23$. So, $n_1 \leq 4$.\\
- If $u_1 \in \{5,7,15\}$, then  $n_1=0$ (from the expressions of $\delta_7$, $\delta_2$ and $\gamma_8$).\\
- It remains the case where $u,v \in \{1,3,5\}$ (with $u \geq 3$ or $v \geq 3$) and $u_1 = 1$. We immediately get: $d_j = \delta_j = 0$ for any $j\geq 1$ and $c_2 =c_3 = \gamma_2 = \gamma_3 = 0$. Thus $A = x^a(x+1)^b {M_1}^{c_1} {M_4}^{c_4} {M_5}^{c_5}$. We may suppose that $u \in \{3,5\}$ and we apply Lemma \ref{omegaleq4}:\\
$\bullet$ If $u= 3$, then $c_4 = \gamma_4 = 0$. So, $\omega(A) \leq 4$, $c_5 =0$ and $A =T_1$, $n =0$, $m=1$.\\
$\bullet$ If $u=5$, then $c_1 = \delta_1 = 0$, $A = x^a(x+1)^b {M_4}^{c_4} {M_5}^{c_5}$. Hence, $A = T_5$ and $n=m=0$.
\end{proof}

\begin{corollary} \label{finalconditions}
If $A$ is perfect, then $u_4 = u_5=1$, $d_j \in \{0,1\}$ for $j \geq 2$ and
$$\begin{array}{l}
u,v \in \{1,3,5,7,9,13,15\}, u_1 \in \{1,3,5,7,15\}, u_2, u_3, v_1 \in \{1,3\},\\
n,m ,n_1 \leq 4, \  n_2, n_3, m_1 \leq 3, \ n_4, n_5 \leq 5.
\end{array}$$
\end{corollary}

\section{Perspectives}
We have established (\cite{Gall-Rahav12}, Conjecture 5.2) a conjecture about the factorization of $\sigma(M^{2h})$, for a Mersenne prime $M$.
\begin{conjecture} \label{oldconj}
For any Mersenne prime $M$ and for any integer $h \geq 2$,  $\sigma(M^{2h})$ is divisible by a non Mersenne prime.
\end{conjecture}
We try to prove it since a few moments (see e.g. \cite{Gall-Rahav14}). It remains the cases where $2h+1$ is divisible only by primes $p$ such that $p \in \{5,7\}$ or
$(p > 7$ is not a Mersenne prime) or ($8$ does not divide the order of $2$ modulo $p$).\\
Conjecture \ref{oldconj} implies that all even non splitting perfect polynomials with only Mersenne as odd prime factors are the nine first known, $T_1, \ldots, T_9$.

In this paper, we are interested in $2$-Mersenne primes which are possible divisors of perfect polynomials. Our study are very limited because of the difficulty to prove irreducibility of such polynomials. One could continue the work in several directions.
For example, we may replace $2$- Mersenne primes ${M_1}^{abc}$ by ${M}^{abc}$, for some another Mersenne prime $M$...
Despite of this kind of difficulty, we would like to state the following two conjectures (the first and Conjecture \ref{oldconj} would imply the second).
\begin{conjecture} \label{newconj1}
For any $2$-Mersenne prime $S$ and for any integer $h \geq 2$,  $\sigma(S^{2h})$ is divisible by a non Mersenne prime or by a non $2$-Mersenne.
\end{conjecture}
\begin{conjecture} \label{newconj2}
The only odd divisors of even non splitting perfect polynomials over $\F_2$ are Mersenne or $2$-Mersenne.
\end{conjecture}

\def\biblio{\def\titrebibliographie{References}\thebibliography}
\let\endbiblio=\endthebibliography

%\def\references{\def\titrebibliographie{R\' ef\'
%erences}\thebibliography}
%\let\endreferences=\endthebibliography

%%%% MACROS DE SEROUL POUR LES REFERENCES %%%%

%%%%%%% bibliographie selon AMS style %%%%%%%%%%%
%%%%%%% inspir de TUGboat 11 (1990), p. 609 %%%%%%%

\newbox\auteurbox
\newbox\titrebox
\newbox\titrelbox
\newbox\editeurbox
\newbox\anneebox
\newbox\anneelbox
\newbox\journalbox
\newbox\volumebox
\newbox\pagesbox
\newbox\diversbox
\newbox\collectionbox
%--------------------------------------------
\def\fabriquebox#1#2{\par\egroup
\setbox#1=\vbox\bgroup \leftskip=0pt \hsize=\maxdimen \noindent#2}
%--------------------------------------------
\def\bibref#1{\bibitem{#1}

%\mbox{}\ignorespaces

\setbox0=\vbox\bgroup}
%--------------------------------------------
\def\auteur{\fabriquebox\auteurbox\styleauteur}
\def\titre{\fabriquebox\titrebox\styletitre}
\def\titrelivre{\fabriquebox\titrelbox\styletitrelivre}
\def\editeur{\fabriquebox\editeurbox\styleediteur}

\def\journal{\fabriquebox\journalbox\stylejournal}

\def\volume{\fabriquebox\volumebox\stylevolume}
\def\collection{\fabriquebox\collectionbox\stylecollection}
%--------------------------------------------
{\catcode`\- =\active\gdef\annee{\fabriquebox\anneebox\catcode`\-
=\active\def -{\hbox{\rm
\string-\string-}}\styleannee\ignorespaces}}
%--------------------------------------------
{\catcode`\-
=\active\gdef\anneelivre{\fabriquebox\anneelbox\catcode`\-=
\active\def-{\hbox{\rm \string-\string-}}\styleanneelivre}}
%--------------------------------------------
{\catcode`\-=\active\gdef\pages{\fabriquebox\pagesbox\catcode`\-
=\active\def -{\hbox{\rm\string-\string-}}\stylepages}}
%--------------------------------------------
{\catcode`\-
=\active\gdef\divers{\fabriquebox\diversbox\catcode`\-=\active
\def-{\hbox{\rm\string-\string-}}\rm}}
%--------------------------------------------
\def\ajoutref#1{\setbox0=\vbox{\unvbox#1\global\setbox1=
\lastbox}\unhbox1 \unskip\unskip\unpenalty}
%--------------------------------------------
\newif\ifpreviousitem
\global\previousitemfalse
\def\separateur{\ifpreviousitem {,\ }\fi}
%--------------------------------------------
\def\voidallboxes
{\setbox0=\box\auteurbox \setbox0=\box\titrebox
\setbox0=\box\titrelbox \setbox0=\box\editeurbox
\setbox0=\box\anneebox \setbox0=\box\anneelbox
\setbox0=\box\journalbox \setbox0=\box\volumebox
\setbox0=\box\pagesbox \setbox0=\box\diversbox
\setbox0=\box\collectionbox \setbox0=\null}
%--------------------------------------------
\def\fabriquelivre
{\ifdim\ht\auteurbox>0pt
\ajoutref\auteurbox\global\previousitemtrue\fi
\ifdim\ht\titrelbox>0pt
\separateur\ajoutref\titrelbox\global\previousitemtrue\fi
\ifdim\ht\collectionbox>0pt
\separateur\ajoutref\collectionbox\global\previousitemtrue\fi
\ifdim\ht\editeurbox>0pt
\separateur\ajoutref\editeurbox\global\previousitemtrue\fi
\ifdim\ht\anneelbox>0pt \separateur \ajoutref\anneelbox
\fi\global\previousitemfalse}
%--------------------------------------------
\def\fabriquearticle
{\ifdim\ht\auteurbox>0pt        \ajoutref\auteurbox
\global\previousitemtrue\fi \ifdim\ht\titrebox>0pt
\separateur\ajoutref\titrebox\global\previousitemtrue\fi
\ifdim\ht\titrelbox>0pt \separateur{\rm in}\
\ajoutref\titrelbox\global \previousitemtrue\fi
\ifdim\ht\journalbox>0pt \separateur
\ajoutref\journalbox\global\previousitemtrue\fi
\ifdim\ht\volumebox>0pt \ \ajoutref\volumebox\fi
\ifdim\ht\anneebox>0pt  \ {\rm(}\ajoutref\anneebox \rm)\fi
\ifdim\ht\pagesbox>0pt
\separateur\ajoutref\pagesbox\fi\global\previousitemfalse}
%--------------------------------------------
\def\fabriquedivers
{\ifdim\ht\auteurbox>0pt
\ajoutref\auteurbox\global\previousitemtrue\fi
\ifdim\ht\diversbox>0pt \separateur\ajoutref\diversbox\fi}
%--------------------------------------------
\def\endbibref
{\egroup \ifdim\ht\journalbox>0pt \fabriquearticle
\else\ifdim\ht\editeurbox>0pt \fabriquelivre
\else\ifdim\ht\diversbox>0pt \fabriquedivers \fi\fi\fi.\voidallboxes}
%--------------------------------------------

\let\styleauteur=\sc
\let\styletitre=\it
\let\styletitrelivre=\sl
\let\stylejournal=\rm
\let\stylevolume=\bf
\let\styleannee=\rm
\let\stylepages=\rm
\let\stylecollection=\rm
\let\styleediteur=\rm
\let\styleanneelivre=\rm

\begin{biblio}{99}

\begin{bibref}{Canaday}
\auteur{E. F. Canaday} \titre{The sum of the divisors of a
polynomial} \journal{Duke Math. J.} \volume{8} \pages 721-737 \annee
1941
\end{bibref}

\begin{bibref}{Gall-Rahav5}
\auteur{L. H. Gallardo, O. Rahavandrainy} \titre{Even perfect
polynomials over $\F_2$ with four prime factors} \journal{Intern. J.
of Pure and Applied Math.} \volume{52(2)} \pages 301-314 \annee 2009
\end{bibref}

\begin{bibref}{Gall-Rahav12}
\auteur{L. H. Gallardo, O. Rahavandrainy} \titre{On even (unitary) perfect
polynomials over $\F_{2}$ } \journal{Finite Fields Appl.} \volume{18} \pages 920-932 \annee 2012
\end{bibref}

\begin{bibref}{Gall-Rahav13}
\auteur{L. H. Gallardo, O. Rahavandrainy} \titre{Characterization of Sporadic perfect
polynomials over $\F_{2}$ } \journal{Functiones et Approx.} \volume{55.1} \pages 7-21 \annee 2016
\end{bibref}

\begin{bibref}{Gall-Rahav14}
\auteur{L. H. Gallardo, O. Rahavandrainy} \titre{On (unitary) perfect polynomials over $\F_2$ with only Mersenne primes as
odd divisors} \journal{arXiv: 1908.00106v1} \annee Jul. 2019
\end{bibref}

\end{biblio}

\end{document}